\documentclass[12pt]{amsart}
\usepackage{amssymb}

\newtheorem{theorem}{Theorem}[section]
\newtheorem{lemma}[theorem]{Lemma}

\newtheorem{proposition}[theorem]{Proposition}

\numberwithin{equation}{section}

\begin{document}

\baselineskip=15.5pt

\title[Equivariant Quillen determinant bundle]{Realization of an equivariant
holomorphic Hermitian line bundle as a Quillen determinant bundle}

\author[I. Biswas]{Indranil Biswas}

\address{School of Mathematics, Tata Institute of Fundamental
Research, Homi Bhabha Road, Bombay 400005, India}

\email{indranil@math.tifr.res.in}



\date{}

\begin{abstract}
Let $M$ be an irreducible smooth complex projective variety equipped with an action
of a compact Lie group $G$, and let $({\mathcal L}\, ,h)$ be a $G$--equivariant
holomorphic Hermitian line bundle on $M$. Given a compact connected Riemann surface
$X$, we construct a $G$--equivariant holomorphic Hermitian line bundle
$(L\, ,H)$ on $X\times M$ (the action of $G$ on $X$ is trivial), such that
the corresponding Quillen determinant line bundle $({\mathcal Q}\, , h_Q)$,
which is a $G$--equivariant holomorphic Hermitian line bundle on $M$,
is isomorphic to the given $G$--equivariant holomorphic Hermitian line bundle
$({\mathcal L}\, ,h)$. This proves a conjecture in \cite{DM}.
\end{abstract}

\maketitle

\section{Introduction}

This work was inspired by \cite{DM} where the following result is proved.
Let $M$ be an irreducible smooth complex projective variety and $\mathcal L$
an ample line bundle on $M$ equipped with a Hermitian structure $h$ of positive
curvature. There is a natural family of Cauchy--Riemann operators on ${\mathbb C}
{\mathbb P}^1$, parametrized by $M$, such that the corresponding
Quillen determinant line bundle, which is a holomorphic Hermitian line bundle
on $M$, is holomorphically isomorphic to a positive tensor
power of $({\mathcal L}\, ,h)$ \cite[p. 785, Theorem 1.1]{DM}. It is conjectured
in \cite{DM} that an equivariant version also holds (see \cite[p. 793, \S~5]{DM}).

Let $M$ be as before. Assume that it is equipped with a
$C^\infty$ action of a compact Lie group $G$ via holomorphic automorphisms of
$M$. Let $({\mathcal L}\, ,h)$ be any $G$--equivariant
holomorphic Hermitian line bundle on $M$. 

Let $X$ be a compact connected Riemann surface equipped with a K\"ahler form
$\omega_X$. Let $L_0$ be a holomorphic line bundle on $X$ of degree
$\text{genus}(X)-2$ such that $H^0(X,\, L_0)\,=\, 0$. Fix a Hermitian structure
$h_0$ on $L_0$.

The action of $G$ on $M$ and the trivial action of $G$ on $X$ together produce
an action of $G$ on $X\times M$. Let $p_1$ and $p_2$ be the projections of $X\times
M$ on $X$ and $M$ respectively. Consider the Hermitian structure
$H\, :=\, (p^*_1 h_0)\otimes (p^*_2h)$ on the holomorphic line bundle
$$
L\, :=\, (p^*_1 L_0)\otimes (p^*_2{\mathcal L})
$$
over $X\times M$. The action of $G$ on $\mathcal L$ and the trivial action of $G$
on $L_0$ together produce an action of $G$ on $L$, thus making $L$ a $G$--equivariant
holomorphic Hermitian line bundle on $X\times M$. We will consider $(L\, ,H)$ as
a family of holomorphic Hermitian line bundles o $X$. Let $({\mathcal Q}\, , h_Q)$
be the Quillen determinant line bundle associated to the triple $(L\, ,\, H\, ,\omega_X)$.
It is a $G$--equivariant holomorphic Hermitian line bundle on $M$.

We prove the following (see Theorem \ref{thm1}):

\textit{The two $G$--equivariant holomorphic Hermitian line bundle on $M$, namely
$({\mathcal L}\, ,h)$ and $({\mathcal Q}\, , h_Q)$, are isomorphic.}

We note that this proves the earlier mentioned conjecture in \cite{DM}.

\section{A holomorphic family and its determinant bundle}

Let $X$ be a compact connected Riemann surface. Let $g$ be the genus of $X$.
Fix a holomorphic line bundle $L_0$ on $X$ of degree $g-2$ such that
\begin{equation}\label{e0}
H^0(X,\, L_0)\,=\, 0\, .
\end{equation}
We note that such a line bundle exists. Indeed, if $g\, \leq\,1$,
then any holomorphic line bundle of degree $g-2$ works; if $g\,=\,2$, then any
nontrivial holomorphic line bundle of degree zero works; if $g\, \geq\, 2$, then
any point outside the image of the Abel-Jacobi map $\text{Sym}^{g-2}(X)\,
\longrightarrow\, \text{Pic}^{g-2}(X)$ works. From Riemann--Roch it follows that
\begin{equation}\label{e1}
\dim H^1(X,\,  L_0)\, =\, 1\, .
\end{equation}
Fix a $C^\infty$ Hermitian structure
$h_0$ on $L_0$. Also, fix a K\"ahler form $\omega_X$ on $X$.

Let $M$ be a connected complex projective manifold (meaning a connected smooth
complex projective variety). Assume that a compact Lie group $G$ acts smoothly on
$M$ via holomorphic automorphisms. Let $({\mathcal L}\, ,h)$ be a $G$--equivariant
holomorphic Hermitian line bundle on $M$. This means that the holomorphic line
bundle $L$ is equipped with an action of $G$ such that
\begin{enumerate}
\item for each element $z\, \in\, G$, the action of $z$ on $\mathcal L$ is a
holomorphic automorphism of the line bundle $\mathcal L$ over the automorphism of
$M$ given by the action of $z$ on $M$,

\item the action of $G$ on $\mathcal L$ is $C^\infty$ and it preserves $h$.
\end{enumerate}

Let $p_1$ and $p_2$ be the projections on $X\times M$ to $X$ and $M$ respectively.
Consider the holomorphic line bundle
$$
L\, :=\, (p^*_1 L_0)\otimes (p^*_2{\mathcal L})\, \longrightarrow\, X\times M\, .
$$
It is equipped with the Hermitian structure
\begin{equation}\label{H}
H\, :=\, (p^*_1 h_0)\otimes (p^*_2h)\, .
\end{equation}
The action of $G$ on $M$ and the trivial action of $G$ on $X$ together define
an action of $G$ on $X\times M$. Similarly, the action of $G$ on $\mathcal L$ and the
trivial action of $G$ on $L_0$ together define an action of $G$ on $L$. This action
of $G$ on $L$ clearly preserves $H$.

Consider $(L\, ,H)$ as a family of holomorphic Hermitian line bundles o $X$ parametrized
by $M$. Let
$$
({\mathcal Q}\, , h_Q)\, \longrightarrow\, M
$$
be the Quillen determinant line bundle associated to $(L\, ,\, H\, ,\omega_X)$
\cite{Qu}. For any point $y\,\in\, M$, let $(L^y\, , H^y)$ be the
holomorphic Hermitian line bundle on $X$ obtained by restricting $(L\, , H)$
to $X\times\{y\}$. Note that $(L^y\, , H^y)$ is isomorphic to the
holomorphic Hermitian line bundle $(L_0\, ,h_0)$.
We recall that the fiber ${\mathcal Q}_y$ is identified with the complex line
$\bigwedge^{\rm top} H^0(X\, , L^y)^*\otimes \bigwedge^{\rm top} H^1(X,\,  L^y)$
\cite{Qu}. In view of \eqref{e0} and \eqref{e1}, the fiber ${\mathcal Q}_y$
is identified with $H^1(X,\,  L^y)$.

The action of $G$ on $L$ produces an action of $G$ on $\mathcal Q$.
The action of any $z\,\in\, G$ on ${\mathcal Q}$ is a holomorphic automorphism
of the line bundle ${\mathcal Q}$ over the automorphism of $M$ given by $z$.
The action of $G$ on ${\mathcal Q}$ preserves the Hermitian structure $h_Q$ on
$\mathcal Q$ because the action of $G$ on $L$ preserves $H$ and
the trivial action on $X$ preserves $\omega_X$.

Let
\begin{equation}\label{e5}
\xi\, :=\, M\times H^1(X,\,  L_0)\, \longrightarrow\, M
\end{equation}
be the holomorphically trivial line bundle with fiber $H^1(X,\,  L_0)$
(see \eqref{e1}). The trivial action of $G$ on $H^1(X,\,  L_0)$ and the action
of $G$ on $M$ together define an action of $G$ on $\xi$. The
actions of $G$ on ${\mathcal L}$ and $\xi$ together produce an action of $G$
on ${\mathcal L}\otimes \xi$ that is a lift of the action of $G$ on $M$.

\begin{lemma}\label{lem1}
The holomorphic line bundle $\mathcal Q$ over $M$ is identified with
${\mathcal L}\otimes \xi$. This identification is $G$--equivariant.
\end{lemma}

\begin{proof}
{}From \eqref{e0} it follows that
$$
R^0p_{2*}L\, =\, 0
$$
(recall that $L^y$ is isomorphic to $L_0$). By the projection formula
\cite[p. 253, Ex.~8.3]{Ha}, we have
$$
R^1p_{2*}L\,=\, {\mathcal L}\otimes  R^1p_{2*}(p^*_1 L_0)\, .
$$
But $R^1p_{2*}(p^*_1 L_0) \,=\, \xi$. Therefore, we get an isomorphism
\begin{equation}\label{e4}
\tau\, :\, \mathcal Q\,:=\, {\rm Det}(L) \,=\, R^1p_{2*}L\,
\stackrel{\sim}{\longrightarrow}\, {\mathcal L}\otimes\xi\, .
\end{equation}
{}From the construction of $\tau$ it follows immediately that it intertwines
the actions of $G$ on $\mathcal Q$ and ${\mathcal L}\otimes\xi$.
\end{proof}

Let ${\nabla}^Q$ be the Chern connection on $\mathcal Q$ for the
Hermitian structure $h_Q$. The curvature of ${\nabla}^Q$ will be denoted
by ${\mathcal K}({\nabla}^Q)$. The curvature ${\mathcal K}({\nabla}^Q)$ can be
computing using \cite{Qu}, \cite{BGS}.

Let $\nabla^{\mathcal L}$ denote the Chern connection for $({\mathcal L}\, ,h)$.
Its curvature will be denoted by ${\mathcal K}(\nabla^{\mathcal L})$.

\begin{proposition}\label{prop1}
The two $(1,1)$--forms ${\mathcal K}({\nabla}^Q)$ and
${\mathcal K}(\nabla^{\mathcal L})$ on $M$ coincide.
\end{proposition}

\begin{proof}
The Chern connection on the holomorphic Hermitian line bundle $(L_0\, , h_0)$
(respectively, $(L\, , H)$) will be denoted by $\nabla^{L_0}$ (respectively,
$\nabla^L$). Let ${\mathcal K}(\nabla^{L_0})$ (respectively, ${\mathcal K}
(\nabla^L)$) be the curvature of $\nabla^{L_0}$ (respectively, $\nabla^L$). From
the definition of $H$ (see \eqref{H}) it follows immediately that
\begin{equation}\label{e2}
{\mathcal K}(\nabla^L)\, =\, p^*_1{\mathcal K}(\nabla^{L_0}) +p^*_2
{\mathcal K}(\nabla^{\mathcal L})\, .
\end{equation}
Let ${\mathcal K}(\omega_X)\,\in\, C^\infty(X;\, \Omega^{1,1}_X)$ be the curvature
of $TX$ for the K\"ahler form $\omega_X$.

A theorem due to Quillen and Bismut--Gilet--Soul\'e says that ${\mathcal K}
({\nabla}^Q)$ is given by the following fiber integral along $X$
\begin{equation}\label{e3}
{\mathcal K}({\nabla}^Q)\,=\, - \frac{1}{2\pi\sqrt{-1}}\left(\int_{(X\times M)
/M}({\mathcal K}(\nabla^L)+\frac{1}{2}{\mathcal K}(\nabla^L)^2)\wedge (1+
\frac{1}{2}p^*_1{\mathcal K}(\omega_X))\right)_2
\end{equation}
\cite[p. 51, Theorem 0.1]{BGS}, \cite{Qu}, where $(\beta)_2$ denote the component
of the differential form $\beta$ of degree two; note that $\frac{1}{4\pi\sqrt{-1}}
{\mathcal K}(\omega_X)$ is the Todd form on $X$ for the K\"ahler form
$\omega_X$ that represents the Todd class $\frac{1}{2} c_1(TX)$.
Using \eqref{e2}, the expression in \eqref{e3} reduces to
\begin{equation}\label{f}
2\pi\sqrt{-1}\cdot {\mathcal K}({\nabla}^Q)\,=\, - {\mathcal K}
(\nabla^{\mathcal L})\cdot \int_X ({\mathcal K}(\nabla^{L_0}) +
\frac{1}{2} {\mathcal K}(\omega_X))\, .
\end{equation}
Now note that
$$
\frac{1}{2\pi\sqrt{-1}}
\int_X ({\mathcal K}(\nabla^{L_0}) + \frac{1}{2} {\mathcal K}(\omega_X))\,=\,
\text{degree}(L_0)+ \frac{1}{2}\text{degree}(TX)\,=\, g-2+1-g\,=\,- 1\, .
$$
Using this, from \eqref{f} we conclude that
${\mathcal K}({\nabla}^Q)\,=\, {\mathcal K}(\nabla^{\mathcal L})$.
\end{proof}

The Hermitian structure $L_0$ and the K\"ahler form $\omega_X$ together
produce an inner product on $H^1(X,\,  L_0)$. This inner product defines
a Hermitian structure $h_\xi$ on the holomorphic line bundle $\xi$ in \eqref{e5}.
Note that the Chern connection on $\xi$ for $h_\xi$ is flat.

The Hermitian structure
$h$ on $\mathcal L$ and the Hermitian structure $h_\xi$ on $\xi$ together produce
a Hermitian structure $\widetilde h$ on ${\mathcal L}\otimes\xi$.

\begin{proposition}\label{prop2}
For the isomorphism $\tau$ in \eqref{e4}, there is a positive real number $t$ such
that $\tau^* \widetilde h\,=\, t\cdot h_Q$.
\end{proposition}

\begin{proof}
There is a real valued $C^\infty$ function $f$ on $M$ such that
$$
\tau^* \widetilde h\,=\, \exp(f)\cdot h_Q\, .
$$
{}From Proposition \ref{prop1} it follows
the two holomorphic Hermitian line bundles $({\mathcal Q}\, , h_Q)$ and $({\mathcal Q}\, ,
\exp(f)\cdot h_Q)$ have the same curvature. This implies that $f$ is a harmonic
function. Since $M$ is compact and connected, any harmonic function on it is a constant
one.
\end{proof}

\begin{theorem}\label{thm1}
The two $G$--equivariant holomorphic Hermitian line bundles $({\mathcal L}\, , h)$
and $({\mathcal Q}\, , h_Q)$ are isomorphic.
\end{theorem}

\begin{proof}
Consider the isomorphism
$$
\frac{1}{\sqrt{t}}\cdot \tau\, :\, \mathcal Q\,
\longrightarrow\, {\mathcal L}\otimes\xi\, ,
$$
where $\tau$ is the isomorphism in \eqref{e4}, and $t$ is the constant in
Proposition \ref{prop2}. From Lemma \ref{lem1} and Proposition \ref{prop2}
it follows immediately that this is a $G$--equivariant holomorphic isomorphism
that takes the Hermitian structure $h_Q$ on $\mathcal Q$ to the Hermitian structure
${\widetilde h}$ on ${\mathcal L}\otimes\xi$.

The two $G$--equivariant holomorphic Hermitian line bundles $({\mathcal L}\, , h)$
and $({\mathcal L}\otimes\xi\, , {\widetilde h})$ are clearly isomorphic.
Therefore, the two $G$--equivariant holomorphic Hermitian line bundles $({\mathcal L}\, ,
h)$ and $({\mathcal Q}\, , h_Q)$ are isomorphic.
\end{proof}

\section*{Acknowledgements}
The author is grateful to the Harish--Chandra Research Institute for hospitality
while the work was carried out. The author acknowledges the support of the J. C.
Bose Fellowship.


\end{document}